\documentclass{amsart}
\usepackage{setspace}
\usepackage{a4}
\usepackage{amssymb,amsmath,amsthm,latexsym}
\usepackage{amsfonts}
\usepackage{amsfonts}
\usepackage{graphicx}
\usepackage{textcomp}
\usepackage{cite}
\usepackage{enumerate}
\usepackage[mathscr]{euscript}
\usepackage{mathtools}
\newtheorem{theorem}{Theorem}[section]

\newtheorem{corollary}[theorem] {Corollary}
\newtheorem{definition}[theorem]{Definition}

\newtheorem{remark}[theorem]{Remark}
\title{This is the title}
\usepackage{amssymb}
\usepackage{amssymb}
\usepackage{amssymb}
\usepackage{amssymb}
\usepackage{amsmath}
\usepackage{tikz}
\usepackage{hyperref}
\usepackage{enumerate}
\usepackage{mathtools}
\usepackage{amsmath}
\usepackage{tikz}
\usepackage{amssymb}
\usepackage{amsmath}
\usepackage{tikz}
\usepackage{hyperref}
\begin{document}
\begin{center}
{\bf{Perturbation of p-approximate Schauder frames for separable Banach spaces}}\\
K. Mahesh Krishna and P. Sam Johnson\\
Department of Mathematical and Computational Sciences\\ 
National Institute of Technology Karnataka (NITK), Surathkal\\
Mangaluru 575 025, India  \\
Emails: kmaheshak@gmail.com,  sam@nitk.edu.in\\

Date: \today
\end{center}

\hrule
\vspace{0.5cm}
%--------------------------------------
\textbf{Abstract}:  Paley-Wiener theorem for frames for Hilbert spaces, Banach frames, Schauder frames and atomic decompositions for Banach spaces are known. In this paper, we derive Paley-Wiener theorem for p-approximate Schauder frames for separable Banach spaces. We show that our results give  Paley-Wiener theorem for frames for Hilbert spaces.



%\tableofcontents

\section{Introduction}
About a century old theorem of Paley and Wiener states that sequences which are close to orthonormal  bases for Hilbert spaces are Riesz  bases (see Chapter 1, Theorem 13 in \cite{YOUNG} and \cite{ARSOVE}). Since frames are generalizations of Riesz bases, we naturally ask whether a sequence which is close to a frame is a frame? Recall that 
a sequence $\{\tau_n\}_n$ in  a separable Hilbert space $\mathcal{H}$ over $\mathbb{K}$ ($\mathbb{R}$ or $\mathbb{C}$) is said to be a
frame for $\mathcal{H}$ if there exist $a,b>0$ such that 
\begin{align*}
a\|h\|^2 \leq \sum_{n=1}^\infty |\langle h, \tau_n\rangle|^2\leq b\|h\|^2, \quad \forall h \in \mathcal{H}.
\end{align*}
Constants $a$ and $b$ are called as lower and upper frame bounds, respectively \cite{DUFFIN}. 
First Paley-Wiener theorem (also known as perturbation theorem) of a frame for a Hilbert space is due to  Christensen, in 1995, which states as follows.
\begin{theorem}\cite{PALEY1}\label{FIRSTPER}
	Let $ \{\tau_n\}_{n=1}^\infty$ be a frame for  $\mathcal{H} $ with bounds $ a$ and $b$. If  $ \{\omega_n\}_{n=1}^\infty$  in $\mathcal{H} $ satisfies
	$$ c \coloneqq\sum_{n=1}^{\infty}\|\tau_n-\omega_n\|^2<a,$$
	then it is  a frame for $\mathcal{H} $  with bounds $a\left(1-\sqrt{\frac{c}{a}}\right)^2 $ and $b\left(1+\sqrt{\frac{c}{b}}\right)^2.$
\end{theorem}
In a short time after the derivation of Theorem \ref{FIRSTPER}, Christensen himself generalized Theorem  \ref{FIRSTPER} and obtained the following result.
\begin{theorem}\cite{PALEY2}\label{SECONDPER}
	Let $ \{\tau_n\}_{n=1}^\infty$ be a frame for  $\mathcal{H} $ with bounds $ a$ and $b$.  If  $ \{\omega_n\}_{n=1}^\infty$  in $\mathcal{H} $ is  such that there exist $ \alpha, \gamma \geq0$ with $\alpha+\frac{\gamma}{\sqrt{a}}< 1 $ and
	$$\left\|\sum_{n=1}^{m}c_n(\tau_n-\omega_n) \right\|\leq \alpha\left\|\sum_{n=1}^{m}c_n\tau_n\right \|+\gamma \left(\sum_{n=1}^{m}|c_n|^2\right)^\frac{1}{2},  \quad\forall c_1,  \dots, c_m \in \mathbb{K}, m=1, \dots, $$
	then it is  a frame for $\mathcal{H} $  with bounds $a\left(1-(\alpha+\frac{\gamma}{\sqrt{a}})\right)^2 $ and $b\left(1+(\alpha+\frac{\gamma}{\sqrt{b}})\right)^2.$
\end{theorem}
Casazza and Christensen  extended the Theorem \ref{SECONDPER} further in 1997, and obtained the next theorem.
\begin{theorem}\cite{PALEY3}\label{OLECAZASSA}
	Let $ \{\tau_n\}_{n=1}^\infty$ be a frame for  $\mathcal{H} $ with bounds $ a$ and $b$.  If  $ \{\omega_n\}_{n=1}^\infty$  in $\mathcal{H} $ is  such that there exist $ \alpha, \beta, \gamma \geq0$ with $ \max\{\alpha+\frac{\gamma}{\sqrt{a}}, \beta\}<1$ and
	$$\left\|\sum_{n=1}^{m}c_n(\tau_n-\omega_n) \right\|\leq \alpha\left\|\sum_{n=1}^{m}c_n\tau_n\right \|+\gamma \left(\sum_{n=1}^{m}|c_n|^2\right)^\frac{1}{2}+\beta\left\|\sum_{n=1}^{m}c_n\omega_n\right \|,   \quad\forall c_1,  \dots, c_m \in \mathbb{K}, m=1, \dots, $$
	then it is a frame for $\mathcal{H} $  with bounds $a\left(1-\frac{\alpha+\beta+\frac{\gamma}{\sqrt{a}}}{1+\beta}\right)^2 $ and $b\left(1+\frac{\alpha+\beta+\frac{\gamma}{\sqrt{b}}}{1-\beta}\right)^2.$
\end{theorem}
After the developments of theories of Banach frames, Schauder frames and atomic decompositions for separable Banach spaces (see \cite{GROCHENIG, CASAZZA, CASAZZAHANLARSON, CASAZZACHRISTENSENSTOEVA}) Paley-Wiener theorems are derived for Banach frames, Schauder frames and atomic decompositions (see \cite{CHRISTENSENHEIL, CHEN, ZHUWANG, JAINKAUSHIKKUMAR1, JAINKAUSHIKKUMAR2, STOEVA}). In this paper, we derive Paley-Wiener theorem for p-ASFs (Theorem \ref{OURPERTURBATION}). We show that our result gives Theorem \ref{OLECAZASSA} for Hilbert spaces (Remark \ref{OURCOROLLARY}).

 \section{Paley-Wiener theorem for p-approximate Schauder frames}\label{SECTIONTWO}
 Let $\mathcal{X}$ be a separable Banach space and $\mathcal{X}^*$ be its dual. In the rest of this paper, $\{e_n\}_n$ denotes the standard Schauder basis for  $\ell^p(\mathbb{N})$, $p \in [1, \infty)$. We now recall the definition of approximate Schauder frames for separable Banach spaces.
\begin{definition}\cite{FREEMANODELL, THOMAS}\label{ASFDEF}
	Let $\{\tau_n\}_n$ be a sequence in  $\mathcal{X}$ and 	$\{f_n\}_n$ be a sequence in  $\mathcal{X}^*.$ The pair $ (\{f_n \}_{n}, \{\tau_n \}_{n}) $ is said to be an approximate Schauder frame (ASF) for $\mathcal{X}$ if 
	\begin{align}\label{ASFEQUA}
	S_{f, \tau}:\mathcal{X}\ni x \mapsto S_{f, \tau}x\coloneqq \sum_{n=1}^\infty
	f_n(x)\tau_n \in
	\mathcal{X}
	\end{align}
	is a well-defined bounded linear, invertible operator.
\end{definition} 
Following \cite{MAHESHJOHNSON}, real  $a,b>0$ satisfying 
\begin{align*}
a\|x\|\leq \left\|\sum_{n=1}^\infty
f_n(x)\tau_n \right\|\leq b\|x\|, \quad \forall x \in  \mathcal{X}
\end{align*} 
are called  as  lower ASF bound and  upper ASF bound, respectively. There is a particular case of ASFs studied by the authors of this paper which contains many important properties of frames for Hilbert spaces (see \cite{MAHESHJOHNSON}).
\begin{definition}\cite{MAHESHJOHNSON}\label{PASFDEF}
	An ASF $ (\{f_n \}_{n}, \{\tau_n \}_{n}) $  for $\mathcal{X}$	is said to be a p-ASF, $p \in [1, \infty)$ if both the maps 
	\begin{align*}
	& \theta_f: \mathcal{X}\ni x \mapsto \theta_f x\coloneqq \{f_n(x)\}_n \in \ell^p(\mathbb{N}) \text{ and } \\
	&\theta_\tau : \ell^p(\mathbb{N}) \ni \{a_n\}_n \mapsto \theta_\tau \{a_n\}_n\coloneqq \sum_{n=1}^\infty a_n\tau_n \in \mathcal{X}
	\end{align*}
	are well-defined bounded linear operators. 
\end{definition}
  In order to derive Paley-Wiener theorem for p-ASFs, we need a generalization of result of   Hilding \cite{HILDING}.
  \begin{theorem}\cite{CASAZZAKALTON, PALEY3}\label{cc1}
  	Let $ \mathcal{X}, \mathcal{Y}$ be Banach spaces, $ U : \mathcal{X}\rightarrow \mathcal{Y}$ be a bounded invertible operator. If  a bounded linear  operator $ V : \mathcal{X}\rightarrow \mathcal{Y}$ is  such that there exist  $ \alpha, \beta \in \left [0, 1  \right )$ with 
  	$$ \|Ux-Vx\|\leq\alpha\|Ux\|+\beta\|Vx\|,\quad \forall x \in  \mathcal{X},$$
  	then $ V $ is  invertible and 
  	
  \begin{align*}
  &\frac{1-\alpha}{1+\beta}\|Ux\|\leq\|Vx\|\leq\frac{1+\alpha}{1-\beta} \|Ux\|, \quad\forall x \in  \mathcal{X}\\
  &\frac{1-\beta}{1+\alpha}\frac{1}{\|U\|}\|y\|\leq\|V^{-1}y\|\leq\frac{1+\beta}{1-\alpha} \|U^{-1}\|\|y\|, \quad\forall y \in  \mathcal{Y}.
  \end{align*} 
  \end{theorem}
In the sequel, the standard Schauder basis for $\ell^p(\mathbb{N})$ is denoted by $\{e_n \}_{n}$.
\begin{theorem}\label{OURPERTURBATION}
Let $ (\{f_n \}_{n}, \{\tau_n \}_{n}) $ be a p-ASF for $\mathcal{X}$. Assume that a collection $\{\tau_n \}_{n} $ in $\mathcal{X}$ and a collection $ \{g_n \}_{n}	$ in $\mathcal{X}^*$  are such that there exist $r,s,t,\alpha, \beta, \gamma \geq 0$ with  $ \max\{\alpha+\gamma\|\theta_f S_{f,\tau}^{-1}\|, \beta,s\}<1$ and 
\begin{align}\label{BEFORE12}
\left\|\sum_{n=1}^{m}(f_n-g_n)(x)e_n\right\|\leq r\left\|\sum_{n=1}^{m}f_n(x)e_n\right \|+t \|x\|+s\left\|\sum_{n=1}^{m}g_n(x)e_n\right \|,  \quad\forall x  \in \mathcal{X}, m=1, \dots,
\end{align}
\begin{align}\label{PEREQUATIONA}
\left\|\sum_{n=1}^{m}c_n(\tau_n-\omega_n)\right\|\leq \alpha\left\|\sum_{n=1}^{m}c_n\tau_n\right \|+\gamma \left(\sum_{n=1}^{m}|c_n|^p\right)^\frac{1}{p}+\beta\left\|\sum_{n=1}^{m}c_n\omega_n\right \|,  \quad\forall c_1,  \dots, c_m \in \mathbb{K}, m=1, \dots.
\end{align}
Then $ (\{g_n \}_{n}, \{\omega_n \}_{n}) $ is a p-ASF for $\mathcal{X}$ with bounds 
\begin{align*}
 \frac{1-(\alpha+\gamma\|\theta_f S_{f,\tau}^{-1}\|)}{(1+\beta)\|S_{f,\tau}^{-1}\|} \quad \text{and} \quad  \left(\frac{1+\alpha}{1-\beta}\|\theta_\tau\|+\frac{\gamma}{1-\beta}\right)\left(\frac{1+r}{1-s}\|\theta_f\|+\frac{t}{1-s}\right).
\end{align*}
\end{theorem}
\begin{proof}
	For $ m=1, \dots,$  for each  $x \in \mathcal{X}$ and for every $c_1,  \dots, c_m \in \mathbb{K}$, 
	 \begin{align*}
	\left\| \sum\limits_{n=1}^mg_n(x)e_n\right\|&\leq \left\| \sum\limits_{n=1}^m(f_n-g_n)(x)e_n\right\|+\left\| \sum\limits_{n=1}^mf_n(x)e_n\right\|\\
	&\leq(1+r)\left\| \sum\limits_{n=1}^mf_n(x)e_n\right\|+s\left\| \sum\limits_{n=1}^mg_n(x)e_n\right\|+t\|x\|
	\end{align*}
	and 
	\begin{align*}
	\left\|\sum_{n=1}^{m}c_n\omega_n\right\|&\leq \left\|\sum_{n=1}^{m}c_n(\tau_n-\omega_n)\right\|+\left\|\sum_{n=1}^{m}c_n\tau_n\right\|\\
	&\leq (1+\alpha)\left\|\sum_{n=1}^{m}c_n\tau_n\right \|+\gamma \left(\sum_{n=1}^{m}|c_n|^p\right)^\frac{1}{p}+\beta\left\|\sum_{n=1}^{m}c_n\omega_n\right \|.
	\end{align*}
	Hence 
		\begin{align*}
	\left\| \sum\limits_{n=1}^mg_n(x)e_n\right\|\leq\frac{1+r}{1-s}\left\| \sum\limits_{n=1}^mf_n(x)e_n\right\|+\frac{t}{1-s}\|x\|,  \quad\forall x  \in \mathcal{X}, m=1, \dots
	\end{align*}
	and 
	\begin{align*}
	\left\| \sum\limits_{n=1}^mc_n\omega_n\right\|\leq\frac{1+\alpha}{1-\beta}\left\| \sum\limits_{n=1}^mc_n\tau_n\right\|+\frac{\gamma}{1-\beta}\left( \sum\limits_{n=1}^m|c_n|^p\right)^\frac{1}{p},  \quad\forall c_1,  \dots, c_m \in \mathbb{K}, m=1, \dots.
	\end{align*}
	Therefore $\theta_g$ and  $\theta_\omega$ are well-defined bounded linear operators with 
	\begin{align*}
	\|\theta_g\|\leq \frac{1+r}{1-s}\|\theta_f\|+\frac{t}{1-s},  \quad \|\theta_\omega\|\leq\frac{1+\alpha}{1-\beta}\|\theta_\tau\|+\frac{\gamma}{1-\beta}.
	\end{align*} 
	Now Equation (\ref{PEREQUATIONA}) gives 
	\begin{align*}
	\left\|\sum_{n=1}^{\infty}c_n(\tau_n-\omega_n)\right\|\leq \alpha\left\|\sum_{n=1}^{\infty}c_n\tau_n\right \|+\gamma \left(\sum_{n=1}^{\infty}|c_n|^p\right)^\frac{1}{p}+\beta\left\|\sum_{n=1}^{\infty}c_n\omega_n\right \|,  \quad\forall  \{c_n\}_n \in \ell^p(\mathbb{N}).
	\end{align*}
	That is 
	\begin{align}\label{PEREQUATIONB}
	\|\theta_\tau \{c_n\}_n-\theta_\omega \{c_n\}_n\|\leq \alpha \|\theta_\tau \{c_n\}_n\|+\gamma\left( \sum\limits_{n=1}^\infty|c_n|^p\right)^\frac{1}{p}+\beta \|\theta_\omega \{c_n\}_n\|, \quad\forall \{c_n\}_n \in \ell^p(\mathbb{N}).
	\end{align}
	By taking $\{c_n\}_n =\{f_n(S_{f,\tau}^{-1}x)\}_n=\theta_fS_{f,\tau}^{-1}x$ in Equation (\ref{PEREQUATIONB}), we get 
	\begin{align*}
	\|\theta_\tau \theta_fS_{f,\tau}^{-1}x-\theta_\omega \theta_fS_{f,\tau}^{-1}x\|\leq \alpha \|\theta_\tau \theta_fS_{f,\tau}^{-1}x\|+\gamma\left( \sum\limits_{n=1}^\infty|f_n(S_{f,\tau}^{-1}x)|^p\right)^\frac{1}{p}+\beta \|\theta_\omega\theta_fS_{f,\tau}^{-1}x\|,  \quad\forall  x \in \mathcal{X}.
	\end{align*}
	That is, 
	\begin{align*}
	\|x-S_{g,\omega}S_{f,\tau}^{-1}x\|&\leq \alpha \| x\|+\gamma\|\theta_fS_{f,\tau}^{-1}x\|+\beta \|S_{g,\omega}S_{f,\tau}^{-1}x\|\\
	&\leq  (\alpha +\gamma\|\theta_fS_{f,\tau}^{-1}\|)\|x\|+\beta \|S_{g,\omega}S_{f,\tau}^{-1}x\|, \quad\forall  x \in \mathcal{X}.
	\end{align*}
Since $ \max\{\alpha+\gamma\|\theta_f S_{f,\tau}^{-1}\|, \beta\}<1$, we can use Theorem \ref{cc1} to get the operator $S_{g,\omega}S_{f,\tau}^{-1}$ is invertible and 
\begin{align*}
\|(S_{g,\omega} S_{f,\tau}^{-1})^{-1}\| \leq \frac{1+\beta}{1-(\alpha+\gamma\|\theta_f S_{f,\tau}^{-1}\|)}.
\end{align*}
 Hence the operator $S_{g,\omega}=(S_{g,\omega}S_{f,\tau}^{-1})S_{f,\tau}$ is invertible. Therefore $ (\{g_n \}_{n}, \{\omega_n \}_{n}) $ is a p-ASF for $\mathcal{X}$. We get the frame bounds from the following calculations:
 \begin{align*}
  &\| S_{g,\omega}^{-1}\|\leq\|S_{f,\tau}^{-1}\|\| S_{f,\tau}S_{g,\omega}^{-1}\| \leq \frac{\|S_{f,\tau}^{-1}\|(1+\beta)}{1-(\alpha+\gamma\|\theta_f S_{f,\tau}^{-1}\|)}\quad\text{ and }\\
  &\|S_{g,\omega}\|\leq \|\theta_\omega\|\|\theta_g\|\leq \left(\frac{1+\alpha}{1-\beta}\|\theta_\tau\|+\frac{\gamma}{1-\beta}\right)\left(\frac{1+r}{1-s}\|\theta_f\|+\frac{t}{1-s}\right).
 \end{align*}
 \end{proof}
\begin{remark}\label{OURCOROLLARY}
	Theorem \ref{OLECAZASSA} is a corollary for Theorem \ref{OURPERTURBATION}. In particular, Theorems \ref{FIRSTPER} and \ref{SECONDPER} are corollaries for Theorem  \ref{OURPERTURBATION}. Indeed, 
let $\{\tau_n\}_n$ be a frame for  $\mathcal{H}$. We define
\begin{align*}
f_n:\mathcal{H} \ni h \mapsto f_n(h)\coloneqq \langle h, \tau_n\rangle \in \mathbb{K}, \quad \forall n \in \mathbb{N}.
\end{align*}	
Then $\theta_f=\theta_\tau$ and  $ (\{f_n \}_{n}, \{\tau_n \}_{n}) $ is  a 2-approximate frame   for $\mathcal{H}$.	We also define 
\begin{align*}
g_n\coloneqq f_n, \quad \forall n \in \mathbb{N}.
\end{align*}
Then condition (\ref{BEFORE12}) holds trivially.
Theorem \ref{OURPERTURBATION} now says that $ (\{g_n \}_{n}, \{\omega_n \}_{n}) $ is a p-ASF for $\mathcal{X}$. To prove Theorem \ref{OLECAZASSA}, it now suffices to prove that $\{\omega_n\}_n$ is a frame for  $\mathcal{H}$. Since $ (\{g_n \}_{n}, \{\omega_n \}_{n}) $ is a p-ASF for $\mathcal{X}$, it follows that $\theta_\omega$ is surjective. Theorem 5.4.1 in \cite{OLEBOOK} now says that $\{\omega_n\}_n$ is a frame for $\mathcal{H}$.
\end{remark}
\begin{corollary}
Let $q$ be the conjugate index of $p$. Let $ (\{f_n \}_{n}, \{\tau_n \}_{n}) $ be a p-ASF for $\mathcal{X}$. Assume that a collection $\{\tau_n \}_{n} $ in $\mathcal{X}$ and a collection $ \{g_n \}_{n}	$ in $\mathcal{X}^*$  are  such that $\sum_{n=1}^{\infty}\|f_n-g_n\|<\infty$
 and
	$$ \lambda \coloneqq \sum_{n=1}^\infty\|\tau_n-\omega_n\|^p <\frac{1}{\|\theta_f S_{f,\tau}^{-1}\|^p}.$$
Then $ (\{g_n \}_{n}, \{\omega_n \}_{n}) $ is a p-ASF for $\mathcal{X}$ with bounds $ \frac{1-\lambda^{1/p}\|\theta_f S_{f,\tau}^{-1}\|}{\|S_{f,\tau}^{-1}\|}$ and $(\|\theta_\tau\|+\lambda^{1/p})(\|\theta_f\|+\sum_{n=1}^{\infty}\|f_n-g_n\|) $.
\end{corollary}
\begin{proof}
	Take $r=0, s=0, t=\sum_{n=1}^{\infty}\|f_n-g_n\|,  \alpha =0, \beta=0, \gamma=\lambda^{1/p}$. Then $ \max\{\alpha+\gamma\|\theta_f S_{f,\tau}^{-1}\|, \beta,s\}<1$ and 
	\begin{align*}
	&\left\|\sum_{n=1}^{m}(f_n-g_n)(x)e_n\right\|\leq \left(\sum_{n=1}^{m}\|f_n-g_n\|\right) \|x\|\leq t\|x\|,   \quad\forall x  \in \mathcal{X}, m=1, \dots,\\
 &\left\|\sum\limits_{n=1}^{m}c_n(\tau_n-\omega_n)\right\|\leq \left(\sum\limits_{n=1}^{m}\|\tau_n-\omega_n\|^q \right)^\frac{1}{q}\left(\sum\limits_{n=1}^{m}|c_n|^p\right)^\frac{1}{p}\leq \gamma\left(\sum\limits_{n=1}^{m}|c_n|^p\right)^\frac{1}{p},  \quad \forall c_1,  \dots, c_m \in \mathbb{K}, m=1, \dots.
	\end{align*}
	By using Theorem \ref{OURPERTURBATION} we now get the result.
\end{proof}
We next derive stability result which does not demand  maximum condition on parameters $\alpha$ and $\gamma$.
\begin{theorem}
Let $ (\{f_n \}_{n}, \{\tau_n \}_{n}) $ be a p-ASF for $\mathcal{X}$. Assume that a collection $\{\tau_n \}_{n} $ in $\mathcal{X}$ and a collection $ \{g_n \}_{n}	$ in $\mathcal{X}^*$  are such that there exist $r,s,t,\alpha, \beta, \gamma \geq 0$ with  $ \max\{ \beta,s\}<1$ and
\begin{align*}
&\left\|\sum_{n=1}^{m}(f_n-g_n)(x)e_n\right\|\leq r\left\|\sum_{n=1}^{m}f_n(x)e_n\right \|+t \|x\|+s\left\|\sum_{n=1}^{m}g_n(x)e_n\right \|,   \quad\forall x  \in \mathcal{X}, m=1, \dots,\\
&\left\|\sum_{n=1}^{m}c_n(\tau_n-\omega_n)\right\|\leq \alpha\left\|\sum_{n=1}^{m}c_n\tau_n\right \|+\gamma \left(\sum_{n=1}^{m}|c_n|^p\right)^\frac{1}{p}+\beta\left\|\sum_{n=1}^{m}c_n\omega_n\right \|,   \quad\forall c_1,  \dots, c_m \in \mathbb{K}, m=1, \dots.
\end{align*}	
Assume that one of the following holds. 
  \begin{enumerate}
  	\item $\sum_{n=1}^{\infty}(\|f_n-g_n\|\|S_{f,\tau}^{-1}\tau_n\|+\|g_n\|\|S_{f,\tau}^{-1}(\tau_n-\omega_n)\|)<1.$
  	\item $\sum_{n=1}^{\infty}(\|f_n-g_n\|\|S_{f,\tau}^{-1}\omega_n\|+\|f_n\|\|S_{f,\tau}^{-1}(\tau_n-\omega_n)\|)<1.$
  	\item $\sum_{n=1}^{\infty}(\|(f_n-g_n)S_{f,\tau}^{-1}\|\|\tau_n\|+\|g_nS_{f,\tau}^{-1}\|\|\tau_n-\omega_n\|)<1.$
  	\item $\sum_{n=1}^{\infty}(\|(f_n-g_n)S_{f,\tau}^{-1}\|\|\omega_n\|+\|f_nS_{f,\tau}^{-1}\|\|\tau_n-\omega_n\|)<1$.
  \end{enumerate}
Then $ (\{g_n \}_{n}, \{\omega_n \}_{n}) $ is a p-ASF for $\mathcal{X}$. Moreover, an upper bound is  
\begin{align*}
	 \left(\frac{1+\alpha}{1-\beta}\|\theta_\tau\|+\frac{\gamma}{1-\beta}\right)\left(\frac{1+r}{1-s}\|\theta_f\|+\frac{t}{1-s}\right).
\end{align*}
\end{theorem}
\begin{proof}
Following the initial lines in the proof of Theorem \ref{OURPERTURBATION}, we see that $\theta_g$ and $\theta_\omega$ are well-defined bounded linear operators. 	We now consider four cases.\\
  Assume (1). Then 
	\begin{align*}
	\left\|x-\sum_{n=1}^{\infty}g_n(x)S_{f,\tau}^{-1}\omega_n\right\|&=\left\|\sum_{n=1}^{\infty}f_n(x)S_{f,\tau}^{-1}\tau_n-\sum_{n=1}^{\infty}g_n(x)S_{f,\tau}^{-1}\omega_n\right\|\leq \sum_{n=1}^{\infty}\|f_n(x)S_{f,\tau}^{-1}\tau_n-g_n(x)S_{f,\tau}^{-1}\omega_n\|\\
	&\leq \sum_{n=1}^{\infty}\bigg\{\|f_n(x)S_{f,\tau}^{-1}\tau_n-g_n(x)S_{f,\tau}^{-1}\tau_n\|+\|g_n(x)S_{f,\tau}^{-1}\tau_n-g_n(x)S_{f,\tau}^{-1}\omega_n\|\bigg\}\\
	&=\sum_{n=1}^{\infty}\bigg\{\|(f_n-g_n)(x)S_{f,\tau}^{-1}\tau_n\|+\|g_n(x)S_{f,\tau}^{-1}(\tau_n-\omega_n)\|\bigg\}\\
	&\leq \left(\sum_{n=1}^{\infty}\bigg\{\|f_n-g_n\|\|S_{f,\tau}^{-1}\tau_n\|+\|g_n\|\|S_{f,\tau}^{-1}(\tau_n-\omega_n)\|\bigg\}\right)\|x\|.
	\end{align*}
	Therefore the operator  $S_{f,\tau}^{-1}S_{g,\omega}$ is invertible.\\
	 Assume (2). Then 
	  
	\begin{align*}
	\left\|x-\sum_{n=1}^{\infty}g_n(x)S_{f,\tau}^{-1}\omega_n\right\|&=\left\|\sum_{n=1}^{\infty}f_n(x)S_{f,\tau}^{-1}\tau_n-\sum_{n=1}^{\infty}g_n(x)S_{f,\tau}^{-1}\omega_n\right\|\leq \sum_{n=1}^{\infty}\|f_n(x)S_{f,\tau}^{-1}\tau_n-g_n(x)S_{f,\tau}^{-1}\omega_n\|\\
	&\leq \sum_{n=1}^{\infty}\bigg\{\|f_n(x)S_{f,\tau}^{-1}\tau_n-f_n(x)S_{f,\tau}^{-1}\omega_n\|+\|f_n(x)S_{f,\tau}^{-1}\omega_n-g_n(x)S_{f,\tau}^{-1}\omega_n\|\bigg\}\\
	&=\sum_{n=1}^{\infty}\bigg\{\|f_n(x)S_{f,\tau}^{-1}(\tau_n-\omega_n)\|+\|(f_n-g_n)(x)S_{f,\tau}^{-1}\omega_n\|\bigg\}\\
	&\leq \left(\sum_{n=1}^{\infty}\bigg\{\|f_n\|\|S_{f,\tau}^{-1}(\tau_n-\omega_n)\|+\|f_n-g_n\|\|S_{f,\tau}^{-1}\omega_n\|\bigg\}\right)\|x\|.
	\end{align*}
	Therefore the operator  $S_{f,\tau}^{-1}S_{g,\omega}$ is invertible.\\
	   Assume (3). Then 
	\begin{align*}
	\left\|x-\sum_{n=1}^{\infty}g_n(S_{f,\tau}^{-1}x)\omega_n\right\|&=\left\|\sum_{n=1}^{\infty}f_n(S_{f,\tau}^{-1}x)\tau_n-\sum_{n=1}^{\infty}g_n(S_{f,\tau}^{-1}x)\omega_n\right\|\leq\sum_{n=1}^{\infty}\|f_n(S_{f,\tau}^{-1}x)\tau_n-g_n(S_{f,\tau}^{-1}x)\omega_n\| \\
	&\leq \sum_{n=1}^{\infty}\bigg\{\|f_n(S_{f,\tau}^{-1}x)\tau_n-g_n(S_{f,\tau}^{-1}x)\tau_n\|+\|g_n(S_{f,\tau}^{-1}x)\tau_n-g_n(S_{f,\tau}^{-1}x)\omega_n\|\bigg\}\\
	&=\sum_{n=1}^{\infty}\bigg\{\|(f_n-g_n)(S_{f,\tau}^{-1}x)\tau_n\|+\|g_n(S_{f,\tau}^{-1}x)(\tau_n-\omega_n)\|\bigg\}\\
	&\leq \left(\sum_{n=1}^{\infty}\bigg\{\|(f_n-g_n)S_{f,\tau}^{-1}\|\|\tau_n\|+\|g_nS_{f,\tau}^{-1}\|\|\tau_n-\omega_n\|\bigg\}\right)\|x\|.
	\end{align*}
	Therefore the operator  $S_{g,\omega}S_{f,\tau}^{-1}$ is invertible.\\
	 Assume (4). Then 
	 \begin{align*}
	\left\|x-\sum_{n=1}^{\infty}g_n(S_{f,\tau}^{-1}x)\omega_n\right\|&=\left\|\sum_{n=1}^{\infty}f_n(S_{f,\tau}^{-1}x)\tau_n-\sum_{n=1}^{\infty}g_n(S_{f,\tau}^{-1}x)\omega_n\right\|\leq\sum_{n=1}^{\infty}\|f_n(S_{f,\tau}^{-1}x)\tau_n-g_n(S_{f,\tau}^{-1}x)\omega_n\|\\
	&\leq \sum_{n=1}^{\infty}\bigg\{\|f_n(S_{f,\tau}^{-1}x)\tau_n-f_n(S_{f,\tau}^{-1}x)\omega_n\|+\|f_n(S_{f,\tau}^{-1}x)\omega_n-g_n(S_{f,\tau}^{-1}x)\omega_n\|\bigg\}\\
	&=\sum_{n=1}^{\infty}\bigg\{\|f_n(S_{f,\tau}^{-1}x)(\tau_n-\omega_n)\|+\|(f_n-g_n)(S_{f,\tau}^{-1}x)\omega_n\|\bigg\}\\
	&\leq \left(\sum_{n=1}^{\infty}\bigg\{\|f_nS_{f,\tau}^{-1}\|\|\tau_n-\omega_n\|+\|(f_n-g_n)S_{f,\tau}^{-1}\|\|\omega_n\|\bigg\}\right)\|x\|.
	\end{align*}
	Therefore the operator  $S_{g,\omega}S_{f,\tau}^{-1}$ is invertible.

Hence in each of assumptions we get that  $ (\{g_n \}_{n}, \{\omega_n \}_{n}) $ is a p-ASF for $\mathcal{X}$.
\end{proof}

  \section{Acknowledgements}
   First author thanks National Institute of Technology Karnataka (NITK) Surathkal for
  financial assistance.

 \bibliographystyle{plain}
 \bibliography{reference.bib}

\begin{thebibliography}{10}

\bibitem{ARSOVE}
Maynard~G. Arsove.
\newblock The {P}aley-{W}iener theorem in metric linear spaces.
\newblock {\em Pacific J. Math.}, 10:365--379, 1960.

\bibitem{CASAZZA}
P.~G. Casazza, S.~J. Dilworth, E.~Odell, Th. Schlumprecht, and A.~Zsak.
\newblock Coefficient quantization for frames in {B}anach spaces.
\newblock {\em J. Math. Anal. Appl.}, 348(1):66--86, 2008.

\bibitem{CASAZZACHRISTENSENSTOEVA}
Pete Casazza, Ole Christensen, and Diana~T. Stoeva.
\newblock Frame expansions in separable {B}anach spaces.
\newblock {\em J. Math. Anal. Appl.}, 307(2):710--723, 2005.

\bibitem{CASAZZAHANLARSON}
Peter~G. Casazza, Deguang Han, and David~R. Larson.
\newblock Frames for {B}anach spaces.
\newblock In {\em The functional and harmonic analysis of wavelets and frames
  ({S}an {A}ntonio, {TX}, 1999)}, volume 247 of {\em Contemp. Math.}, pages
  149--182. Amer. Math. Soc., Providence, RI, 1999.

\bibitem{CASAZZAKALTON}
Peter~G. Casazza and Nigel~J. Kalton.
\newblock Generalizing the {P}aley-{W}iener perturbation theory for {B}anach
  spaces.
\newblock {\em Proc. Amer. Math. Soc.}, 127(2):519--527, 1999.

\bibitem{PALEY3}
Peter~G. Cazassa and Ole Christensen.
\newblock Perturbation of operators and applications to frame theory.
\newblock {\em J. Fourier Anal. Appl.}, 3(5):543--557, 1997.

\bibitem{CHEN}
Dong~Yang Chen, Lei Li, and Ben~Tuo Zheng.
\newblock Perturbations of frames.
\newblock {\em Acta Math. Sin. (Engl. Ser.)}, 30(7):1089--1108, 2014.

\bibitem{PALEY1}
Ole Christensen.
\newblock Frame perturbations.
\newblock {\em Proc. Amer. Math. Soc.}, 123(4):1217--1220, 1995.

\bibitem{PALEY2}
Ole Christensen.
\newblock A {P}aley-{W}iener theorem for frames.
\newblock {\em Proc. Amer. Math. Soc.}, 123(7):2199--2201, 1995.

\bibitem{OLEBOOK}
Ole Christensen.
\newblock {\em Frames and bases: An introductory course}.
\newblock Applied and Numerical Harmonic Analysis. Birkh\"{a}user Boston, Inc.,
  Boston, MA, 2008.

\bibitem{CHRISTENSENHEIL}
Ole Christensen and Christopher Heil.
\newblock Perturbations of {B}anach frames and atomic decompositions.
\newblock {\em Math. Nachr.}, 185:33--47, 1997.

\bibitem{DUFFIN}
R.~J. Duffin and A.~C. Schaeffer.
\newblock A class of nonharmonic {F}ourier series.
\newblock {\em Trans. Amer. Math. Soc.}, 72:341--366, 1952.

\bibitem{FREEMANODELL}
D.~Freeman, E.~Odell, Th. Schlumprecht, and A.~Zs\'{a}k.
\newblock Unconditional structures of translates for {$L_p(\mathbb{R}^d)$}.
\newblock {\em Israel J. Math.}, 203(1):189--209, 2014.

\bibitem{GROCHENIG}
Karlheinz Grochenig.
\newblock Describing functions: atomic decompositions versus frames.
\newblock {\em Monatsh. Math.}, 112(1):1--42, 1991.

\bibitem{HILDING}
Sven~H. Hilding.
\newblock Note on completeness theorems of {P}aley-{W}iener type.
\newblock {\em Ann. of Math. (2)}, 49:953--955, 1948.

\bibitem{JAINKAUSHIKKUMAR2}
P.~K. Jain, S.~K. Kaushik, and L.~K. Vashisht.
\newblock On perturbation of {B}anach frames.
\newblock {\em Int. J. Wavelets Multiresolut. Inf. Process.}, 4(3):559--565,
  2006.

\bibitem{JAINKAUSHIKKUMAR1}
Pawan~Kumar Jain, Shiv~Kumar Kaushik, and Lalit~Kumar Vashisht.
\newblock On stability of {B}anach frames.
\newblock {\em Bull. Korean Math. Soc.}, 44(1):73--81, 2007.

\bibitem{MAHESHJOHNSON}
K.~Mahesh Krishna and P.~Sam Johnson.
\newblock Towards characterizations of approximate {S}chauder frame and its
  duals for {B}anach spaces, ar{X}iv:2010.10514v1 [math.{FA}] 20 oct 2020.
\newblock \textit{Journal of Pseudo-Differential Operators and Applications}
  (accepted for publication).

\bibitem{STOEVA}
Diana~T. Stoeva.
\newblock Perturbation of frames in {B}anach spaces.
\newblock {\em Asian-Eur. J. Math.}, 5(1):1250011, 15, 2012.

\bibitem{THOMAS}
S.~M. Thomas.
\newblock Approximate {S}chauder frames for {$\mathbb{R}^n$, Masters Thesis,
  St. Louis University, St. Louis, MO}.
\newblock 2012.

\bibitem{YOUNG}
Robert~M. Young.
\newblock {\em An introduction to nonharmonic {F}ourier series}, volume~93 of
  {\em Pure and Applied Mathematics}.
\newblock Academic Press, Inc., New York-London, 1980.

\bibitem{ZHUWANG}
Yu~Can Zhu and Si~Yuan Wang.
\newblock The stability of {B}anach frames in {B}anach spaces.
\newblock {\em Acta Math. Sin. (Engl. Ser.)}, 26(12):2369--2376, 2010.

\end{thebibliography}

\end{document}